\documentclass[11pt]{article}
\usepackage{amssymb,a4wide,latexsym,makeidx,epsfig,fleqn}
\usepackage{amsthm}
\usepackage{amsmath}
\usepackage{bbm}
\usepackage{mathrsfs}
\usepackage{graphicx}
\usepackage{appendix}
\usepackage{subfigure}
\usepackage{hyperref}
\usepackage{float}
\usepackage{enumerate}
\usepackage{arydshln}
\newtheorem{theorem}{Theorem}[section]
\newtheorem{lemma}{Lemma}[section]

\newtheorem{definition}{Definition}[section]
\newtheorem{conjecture}{Conjecture}[section]

\newtheorem{case}{Case}[section]

\newtheorem{proposition}{Proposition}[section]

\newtheorem{claim}{Claim}[section]

\DeclareMathOperator{\ex}{ex}

\begin{document}
\textwidth 150mm \textheight 225mm
\title{Connected Tur\'{a}n numbers for Berge paths in hypergraphs\thanks{Email addresses: lpzhangmath@163.com (L.-P. Zhang), h.j.broersma@utwente.nl (h.j.broersma), gyori.ervin@renyi.hu (E.~Gy\H{o}ri), ctompkins496@gmail.com (C.~Tompkins), lgwangmath@163.com (L. Wang)}}
\author{{Lin-Peng Zhang\textsuperscript{a,b}, Hajo Broersma\textsuperscript{b}, Ervin Gy\H{o}ri\textsuperscript{c}, Casey Tompkins\textsuperscript{c}, Ligong Wang\textsuperscript{a,d}}\\
{\small \textsuperscript{a} School of Mathematics and Statistics,}\\
{\small Northwestern Polytechnical University, Xi'an, Shaanxi 710129, P.R. China.}\\
{\small \textsuperscript{b} Faculty of Electrical Engineering, Mathematics and Computer Science,}\\
{\small  University of Twente, P.O. Box 217, 7500 AE Enschede, the Netherlands.}\\
{\small \textsuperscript{c} Alfr\'{e}d R\'{e}nyi Institute of Mathematics,}\\
{\small Hungarian Academy of Sciences, 1053 Budapest, Hungary.}\\
{\small \textsuperscript{d} Xi’an-Budapest Joint Research Center for Combinatorics,}\\
{\small Northwestern Polytechnical University, Xi’an, 710072, China.}}
\date{}
\maketitle
\begin{center}
\begin{minipage}{135mm}
\vskip 0.3cm
\begin{center}
{\small {\bf Abstract}}
\end{center}
{\small Let $\mathcal{F}$ be a family of $r$-uniform hypergraphs. Denote by $\ex^{\mathrm{conn}}_r(n,\mathcal{F})$ the maximum number of hyperedges in an $n$-vertex connected $r$-uniform hypergraph which contains no member of $\mathcal{F}$ as a subhypergraph. 
Denote by $\mathcal{B}C_k$ the Berge cycle of length $k$, and 
by $\mathcal{B}P_k$ the Berge path  of length $k$. F\"{u}redi, Kostochka and Luo, and independently Gy\H{o}ri, Salia and Zamora determined  $\ex^{\mathrm{conn}}_r(n,\mathcal{B}P_k)$ provided $k$ is large enough compared to $r$ and $n$ is sufficiently large. For the case $k\le r$, Kostochka and Luo obtained an upper bound for $\ex^{\mathrm{conn}}_r(n,\mathcal{B}P_k)$.
In this paper, we continue investigating the case $k\le r$.
We precisely determine $\ex^{\mathrm{conn}}_r(n,\mathcal{B}P_k)$ when $n$ is sufficiently large and $n$ is not a multiple of~$r$. For the case $k=r+1$, we determine $\ex^{\mathrm{conn}}_r(n,\mathcal{B}P_k)$ asymptotically. 
 }
\end{minipage}
\end{center}
\date{}
\maketitle

\section{Introduction}
Let $\mathcal{F}$ be a family of $r$-uniform hypergraphs.
An $r$-uniform hypergraph $\mathcal{H}$ is called $\mathcal{F}$-free if it does not contain any member of $\mathcal{F}$ as its subhypergraph.
The Tur\'{a}n number of $\mathcal{F}$, denoted by $\ex_r(n,\mathcal{F})$, is the maximum number of hyperedges in an $n$-vertex $\mathcal{F}$-free $r$-uniform hypergraph.
When $\mathcal{F}$ contains only one hypergraph $F$, we write $\ex_r(n,F)$, and for $r=2$, we simply write $\ex(n,F)$. 
For the graph case, when $G$ is not bipartite the asymptotic behavior of $\ex(n,F)$ follows from the famous Erd\H{o}s-Stone-Simonovits Theorem~\cite{ES-1,ES-2}.
However, for the case when $F$ is bipartite, much less about the $\ex(n,F)$ is known (for reference, see the survey \cite{FS1}).

In this paper, we are interested in paths and cycles.
Denote by $P_k$ the path of length $k$, by $C_k$ the cycle of length $k$ and by $\mathcal{C}_{\ge k}$ the family of cycles of length at least $k$. 
In 1959, Erd\H{o}s and Gallai~\cite{EG} proved the following results on $\ex(n,P_k)$ and $\ex(n,\mathcal{C}_{\ge k})$.

\begin{theorem}[Erd\H{o}s and Gallai, \cite{EG}]\label{EG}
Fix integers $n$ and $k$ such that $n\ge k\ge 1$. 
Then
$\ex(n,P_k)\le \frac{(k-1)n}{2}$,
with equality holding if and only if $G$ is the disjoint union of complete graphs on $k$ vertices.
\end{theorem}

\begin{theorem}[Erd\H{o}s and Gallai, \cite{EG}]
Fix integers $n$ and $k$ such that $n\ge k\ge 3$.
Then,
$\ex(n,\mathcal{C}_{\ge k})\le \frac{(k-1)(n-1)}{2}$.
\end{theorem}

Later, Faudree and Schelp \cite{FS} determined $\ex(n, P_k)$ for all $n$ and $k$, satisfying $n\ge k\ge 1$, with the corresponding extremal graphs. 
Note that the extremal graph in Theorem~\ref{EG} is not connected. 
By considering connected graphs, in 1977, Kopylov \cite{Kopylov} determined $\ex^{\mathrm{conn}}(n, P_k)$, where $\ex^{\mathrm{conn}}(n, P_k)$ denotes the classical Tur\'{a}n number
for connected graphs.

In 2008, Balister, Gy\H{o}ri, Lehel and Schelp~\cite{BGLS} improved Kopylov’s result by characterizing the extremal graphs for all $n$. 
The stability version of $\ex^{\mathrm{conn}}(n, P_k)$ was proved by F\"{u}redi, Kostochka and Verstra\"{e}te~\cite{FKV}. 
To state their result, we need some further notation. Denote by $K_k$ and $E_k$ the complete graph and empty graph on $k$ vertices, respectively. 
Let $G$ and $H$ be two disjoint graphs.
Denote by $G\cup H$ the union of $G$ and $H$, which is the graph with vertex set $V(G)\cup V(H)$ and edge set $E(G)\cup E(H)$. 
Denote by $G+H$ the join of $G$ and $H$, which is the graph with vertex set $V(G)\cup V(H)$ and edge set $E(G)\cup E(H)\cup \{uv:u\in V(G),v\in V(H)\}$.
\begin{theorem}[Kopylov, \cite{Kopylov}, Balister et al., \cite{BGLS}]
Fix integers $n\ge k\ge 4$. Then, 
\[
\ex^{\mathrm{conn}}(n,P_k)=\max\bigg\{\binom{k-1}{2}+(n-k+1), \binom{\left\lceil\frac{k+1}{2}\right\rceil}{2}+\bigg(\left\lfloor\frac{k-1}{2}\right\rfloor\bigg)\bigg(n-\left\lceil\frac{k+1}{2}\right\rceil\bigg)\bigg\}.
\]
The extremal graphs are $(K_{k-2}\cup E_{n-k+1})+K_1$, $(K_1\cup E_{n-\frac{k+1}{2}})+K_{\frac{k-1}{2}}$ for odd $k$ and $(K_2\cup E_{n-\frac{k}{2}-1})+K_{\frac{k}{2}-1}$ for even $k$.
\end{theorem}  

It is natural to study the Tur\'{a}n numbers for paths and cycles in an $r$-uniform hypergraph. 
Note that there are several ways to define paths and cycles in $r$-uniform hypergraphs for $r\ge 3$. 
In this paper, we consider the definition due to Berge.
\begin{definition}
A Berge cycle $\mathcal{B}C_k$ of length $k$ in a hypergraph is a set of $k$ distinct vertices $\{v_1,v_2,\ldots,v_{k}\}$
and $k$ distinct edges $\{e_1,e_2,\ldots,e_k\}$ such that $\{v_i,v_{i+1}\}\subset e_i$
with indices taken modulo $k$. 
The vertices $v_1, v_2, \ldots, v_{k}$
are called the defining vertices, and the hyperedges $e_1, e_2, \ldots, e_k$ are called the defining hyperedges of the Berge cycle.
\end{definition}
\begin{definition}
A Berge path $\mathcal{B}P_k$ of length $k$ in a hypergraph is a set of $k+1$ distinct vertices $\{v_1,v_2,\ldots,v_{k+1}\}$
and $k$ distinct edges $\{e_1,e_2,\ldots,e_k\}$ such that $\{v_i,v_{i+1}\}\subset e_i$ for all $1\le i\le k-1$. The vertices $v_1, v_2, \ldots, v_{k+1}$
are called the defining vertices and the hyperedges $e_1, e_2, \ldots, e_k$ are called the defining hyperedges of the Berge~path.  
\end{definition}
In particular, when the hypergraph under consideration is $2$-uniform, containing a Berge cycle (Berge path) $\mathcal{B}C_k$ ($\mathcal{B}P_k$) is equivalent to containing a cycle $C_k$ (path $P_k$).  
Denote by $\mathcal{BC}_{\ge k}$ the family of Berge cycles of length at least $k$.  

The study of the Tur\'{a}n numbers $\ex_r(n,\mathcal{B}P_k)$ was initiated in 2016 by Gy\H{o}ri, Katona and Lemons~\cite{GKL}, who determined these numbers in the cases $k\le r$ and $k>r+1$.
It turned out that the extremal hypergraphs behave very differently when $k\le r$ and $k>r+1$. 
\begin{theorem}[Gy\H{o}ri, Katona and Lemons, \cite{GKL}]\label{thmGKL}
Fix integers $k$ and $r$ such that $k>r+1>3$. Then 
\[\ex_r(n,\mathcal{B}P_k)\le \frac{n}{k}\binom{k}{r}.\] 
Equality holds if and only if $k|n$, and the only extremal $r$-uniform hypergraph is the disjoint union of $\frac{n}{k}$
copies of the complete $k$-vertex $r$-uniform hypergraph.
Fix integers $k$ and $r$ such that $r\ge k>2$. Then 
\[\ex_r(n,\mathcal{B}P_k)\le \frac{n}{r+1}(k-1).\]
Equality holds if and only if $(r+1)|n$, and the extremal $r$-uniform hypergraph is the disjoint union of $\frac{n}{r+1}$ sets of
size $r+1$ containing $k-1$ hyperedges each.
\end{theorem}

Not much later, Davoodi, Gy\H{o}ri, Methuku, and Tompkins \cite{DGMT} showed that $ex_r(n,P_{r+1})\le n$, which matches the above upper bound of the case $k>r+1$.

Similarly, the function $\ex_r(n,\mathcal{BC}_{\ge k})$ behaves very differently when $k\ge r+2$ and $k\le r+1$ with an exceptional third case when $k=r$. 
F\"{u}redi, Kostochka and Luo~\cite{FKL19,FKL21} initiated the study of $\ex_r(n,\mathcal{BC}_{\ge k})$. 
In~\cite{FKL19}, they obtained sharp bounds and extremal constructions for $k\ge r+3\ge 6$ and infinitely many $n$. 
In~\cite{FKL21}, they obtained the exact bounds and extremal constructions for $k\ge r+4$ and all $n$. 
Later for the case $k\le r-1$, Kostochka and Luo~\cite{KL} obtained a sharp upper bound for infinitely many $n$. 
For the cases $k=r+1$ and $k=r+2$, Ergemlidze, Gy\H{o}ri, Methuku, Salia, Tompkins and Zamora~\cite{EGMSTZ} obtained sharp bounds.
In 2021, Gy\H{o}ri, Lemons, Salia and Zamora \cite{GLSZ} provided a simple proof of Kostochka and Luo's result of \cite{KL} when $k<r$, and they also determined $\ex_r(n,\mathcal{BC}_{\ge r})$.

\begin{theorem}[F\"{u}redi, Kostochka and Luo, \cite{FKL19,FKL21}]
Fix integers $k$ and $r$ such that $k\ge r+3\ge 6$. 
Then,
\[\ex_r(n,\mathcal{BC}_{\ge k})\le \frac{n-1}{k-2}\binom{k-1}{r}.\]
\end{theorem}

\begin{theorem}[Ergemlidze et al., 
\cite{EGMSTZ}]
Fix integers $r$ and $k$ such that $k\ge 4$. 
If $k=r+1$, then $\ex_r(n,\mathcal{BC}_{\ge r+1})\le n-1$. If $k=r+2$, then $\ex_r(n,\mathcal{BC}_{\ge r+2})\le \frac{(n-1)(r+1)}{r}$.  
\end{theorem}

\begin{theorem}[Kostochka and Luo, \cite{KL}]
Fix integers $k$ and $r$ such that $r\ge k+1\ge 5$. 
Let $\mathcal{H}$ be an $n$-vertex $r$-uniform multi-hypergraph, each edge of which has multiplicity at most $k-2$.
If $\mathcal{H}$ is $\mathcal{BC}_{\ge k}$-free, then $e(\mathcal{H})\le \frac{(k-1)(n-1)}{r}$.
\end{theorem}

\begin{theorem}[Gy\H{o}ri et al., \cite{GLSZ}]
Fix integers $k$ and $r$ such that $r>k\ge 3$. 
Then,
\[
\ex_r(n,\mathcal{BC}_{\ge k})=(k-1)\left\lfloor\frac{n-1}{r}\right\rfloor+\mathbbm{1}_{r|n},
\]
where $\mathbbm{1}_{r|n}=1$ if $r|n$ and $\mathbbm{1}_{r|n}=0$ otherwise. 
\end{theorem}

\begin{theorem}[Gy\H{o}ri et al., \cite{GLSZ}]\label{thmcycle}
Let $r\ge 3$ and $n$ be positive integers. 
Then,
\[
\ex_r(n,\mathcal{BC}_{\ge r})=\max\bigg\{(r-1)\left\lfloor\frac{n-1}{r}\right\rfloor,~n-r+1\bigg\}.
\]
\end{theorem}

\begin{theorem}[Gy\H{o}ri et al., \cite{GLSZ}]\label{multi}
Fix integers $n,k$ and $r$ such that $2\le k\le r$. Let $\mathcal{H}$ be an $n$-vertex $r$-uniform multi-hypergraph which is $\mathcal{BC}_{\ge k}$-free. 
Then,
\[
e(\mathcal{H})\le (k-1)\left\lfloor\frac{n-1}{r-1}\right\rfloor.
\]
\end{theorem}

Analogously to graphs, a hypergraph is connected if for any two of its vertices, there is a Berge path containing both vertices.
Note that the extremal $r$-uniform hypergraph when forbidding Berge paths is not connected in general. 
Let $\mathcal{F}$ be a family of $r$-uniform hypergraphs. Denote by $\ex^{\mathrm{conn}}_r(n, \mathcal{F})$ the maximum number of hyperedges in an $n$-vertex connected $\mathcal{F}$-free $r$-uniform hypergraph.
In 2018, Gy\H{o}ri, Methuku, Salia, Tompkins and Vizer~\cite{GMSTV} determined the asymptotics of $\ex^{\mathrm{conn}}_r(n, \mathcal{B}P_k)$. For the case $k\ge 4r\ge 12$, F\"{u}redi, Kostochka and Luo \cite{FKL} determined $\ex^{\mathrm{conn}}_r(n, \mathcal{B}P_k)$ when $n$ is large enough.
Independently in a different range, these numbers were also determined by Gy\H{o}ri, Salia and Zamora \cite{GSZ21}, who also proved the uniqueness of the extremal hypergraphs.
Recently, Gerbner, Nagy, Patk\'{o}s, Salia and Vizer~\cite{GNPSV} obtained a stability version of Gy\H{o}ri et al's result from~\cite{GSZ21}.

\begin{theorem}[Gy\H{o}ri et al., \cite{GMSTV}]
Let $\mathcal{H}_{n,k}$ be a connected $r$-uniform $n$-vertex hypergraph of maximum size with no Berge paths of length~$k$. Then,
\[
\lim_{k \rightarrow \infty}\lim_{n \rightarrow \infty} \frac{|E(\mathcal{H}_{n,k})|}{k^{r-1}n}=\frac{1}{2^{r-1}(r-1)!}.
\]
\end{theorem}

\begin{theorem}[F\"{u}redi, Kostochka and Luo, \cite{FKL}]
Let $n\ge n'_{k,r}\ge k\ge 4r\ge 12$. Then,
\[
\ex^{\mathrm{conn}}_r(n,\mathcal{B}P_k)\le \binom{\left\lceil\frac{k+1}{2}\right\rceil}{r}+\bigg(n-\left\lceil\frac{k+1}{2}\right\rceil\bigg)\binom{\left\lfloor\frac{k-1}{2}\right\rfloor}{r-1}.
\]
\end{theorem}

\begin{theorem}[Gy\H{o}ri, Salia and Zamora, \cite{GSZ21}]
For all integers $n, k$ and $r$ there exists an $N_{k,r}$ such that for $n>N_{k,r}$ and $k\ge 2r+13\ge 18$, 
\[
\ex^{\mathrm{conn}}_r(n,\mathcal{B}P_k)=\binom{\left\lfloor\frac{k-1}{2}\right\rfloor}{r-1}\bigg(n-\left\lfloor\frac{k-1}{2}\right\rfloor\bigg)+\binom{\left\lfloor\frac{k-1}{2}\right\rfloor}{r}+\mathbbm{1}_{2|k}\binom{\left\lfloor\frac{k-1}{2}\right\rfloor}{r-2}.
\]
The extremal hypergraph is unique.
\end{theorem}

For the case $k\le r$, Kostochka and Luo~\cite{KL} obtained an upper bound.  
\begin{theorem}[Kostochka and Luo, \cite{KL}]\label{smallk}
Fix $r\ge k\ge 3$. Then,
\[
\ex^{\mathrm{conn}}_r(n,\mathcal{B}P_k)\le \max\bigg\{k-1,\frac{k}{2r-k+4}n\bigg\}.
\]
\end{theorem}

The remainder of this paper is structured in the following way.
In Section~\ref{results} we present the new contributions of this paper.
In Section~\ref{proofs}, we give the proofs of our main results. 
In Section~\ref{sec4}, we present our postponed proof of Lemma~\ref{lemma20}.
Finally, in Section~\ref{seccr}, we conclude this paper with some final remarks.

\section{Results}\label{results}
In this paper, we present some improvements of Theorem~\ref{smallk}. 
First, we discuss the situation for some small values of $k$.
For $k=2$, we observe that $\ex^{\mathrm{conn}}_r(n,\mathcal{B}P_2)=1$ when $n=r$ and it is not defined for larger $n$.
In the following, we consider the cases $k=3$, $k=4$ and $k\ge 5$, respectively.

\begin{proposition}\label{prop1}
Fix $n\ge r\ge 3$. If $n\le 2r-2$, then $\ex^{\mathrm{conn}}_r(n,\mathcal{B}P_3)=2$ and the extremal hypergraphs consist of two hyperedges sharing at least two common vertices.
If $n\ge 2r-1$ and $n-1$ is a multiple of $r-1$, then $\ex^{\mathrm{conn}}_r(n,\mathcal{B}P_3)=\frac{n-1}{r-1}$ and the extremal hypergraph is a star hypergraph composed of $\frac{n-1}{r-1}$ hyperedges sharing one common vertex.
For the case $n\ge 2r-1$ and $n-1$ is not a multiple of $r-1$, the function $\ex^{\mathrm{conn}}_r(n,\mathcal{B}P_3)$ is not defined.
\end{proposition}
\begin{proposition}\label{prop2}
Fix $n\ge r\ge 4$. If $n\le r+4$, then
\[
\ex^{\mathrm{conn}}_r(n,\mathcal{B}P_4)=4.
\]
If $n\ge r+5$, then
\[
\ex^{\mathrm{conn}}_r(n,\mathcal{B}P_4)\le \max\bigg\{\frac{n-5}{r-1}+3,\frac{n-4}{r-2}+2\bigg\}.
\]
Moreover, if $n\ge r+5$, then $\ex^{\mathrm{conn}}_r(n,\mathcal{B}P_4)$ is not defined if $n-5$ is not a multiple of $r-1$ and $n-4$ is not a multiple of $r-2$.
\end{proposition}

\begin{theorem}\label{thm1}
Fix integers $n,r,k$ with $r\ge k\ge 5$.  
When $n$ is sufficiently large and $n$ is not a multiple of $r$, we have
\[
ex^{\mathrm{conn}}_r(n,\mathcal{B}P_k)=\left\lfloor\frac{k-1}{2}\right\rfloor\left\lfloor\frac{n-1}{r}\right\rfloor+\mathbbm{1}_{2|k},
\]
where $\mathbbm{1}_{2|k}=1$ if $2|k$ and $\mathbbm{1}_{2|k}=0$ otherwise. 
\end{theorem}

\begin{lemma}\label{thm3}
Fix integers $n,r,k$ with $n\ge r>k\ge 3$. 
Let $\mathcal{H}$ be a connected $n$-vertex $\mathcal{B}P_k$-free $r$-uniform multi-hypergraph. 
If $r\le n\le 2r-2$, we have $e(\mathcal{H})\le k-1$.
If $n\ge 2r-1$ and $n$ is not a multiple of $r-1$, we have
\begin{equation*}
e(\mathcal{H})\le \max\bigg\{k-1+\frac{n-k+1}{r-\left\lfloor\frac{k-1}{2}\right\rfloor},\left\lfloor\frac{n-1}{r-1}\right\rfloor\left\lfloor\frac{k-1}{2}\right\rfloor+\mathbbm{1}_{2|k}\bigg\}.
\end{equation*}
\end{lemma}


For the case $k=r+1$, we have the following result.
\begin{theorem}\label{thm2}
Let $r\ge 3$. 
When $n$ is sufficiently large, we have
\[
\ex^{\mathrm{conn}}_r(n,\mathcal{B}P_{r+1})=n-r+1.
\]
\end{theorem}

To prove Theorems \ref{thm1} and \ref{thm2}, we need the following key lemma, the proof of which we postpone to Section~\ref{sec4}. 
Before stating this lemma, we introduce the following 
definition which we adopted from~\cite{GLSZ}. 
\begin{definition}
Let $\mathcal{H}$ be an $r$-uniform hypergraph and let $S\subseteq V(\mathcal{H})$ be a vertex subset of $\mathcal{H}$. 
Denote by
\[
N_{\mathcal{H}}(S)=\{h\in E(\mathcal{H})|h\cap S\neq \emptyset\}
\]
the hyperedge neighborhood of $S$, that is the set of all hyperedges that are incident with at least one vertex of $S$.
\end{definition}

\begin{lemma}\label{lemma20}
Let $r,k,n$ and $m$ be positive integers with $r\ge k\ge 3$, and let $\mathcal{H}$ be an $n$-vertex connected $r$-uniform $\mathcal{B}P_k$-free hypergraph such that every hyperedge has multiplicity at most $m$. 
Assume that a longest Berge path in $\mathcal{H}$ has length $t$, and that $\mathcal{H}$ is $\mathcal{B}C_{t}$-free.
Then at least one of the following holds.
\begin{item}
($i$) There exists $S\subseteq V(\mathcal{H})$ of size $2r-2$ such that $|N_{\mathcal{H}}(S)|\le m+1$.\\
($ii$) There exists $S\subseteq V(\mathcal{H})$ of size at least $2r-1$ such that $|N_{\mathcal{H}}(S)|\le t\le k-1$.
\end{item}
\end{lemma}

\section{Proofs of our main results}\label{proofs}
\begin{proof}[Proof of Proposition \ref{prop1}]
Assume that $\mathcal{H}$ is an $n$-vertex connected $r$-uniform hypergraph containing no Berge path of length 3. Let $e\in E(\mathcal{H})$.
If there exists another hyperedge $f\in E(\mathcal{H})$ such that $|e\cap f|=0$, then by the connectivity of $\mathcal{H}$ there must be a Berge path of length~$3$, a contradiction.
Thus, every hyperedge of $\mathcal{H}$ meets $e$. 
Let $f,g\in E(\mathcal{H})\setminus \{e\}$.
If $f$ and $g$ meet outside of $e$, then $e,f,g$ forms a Berge path of length~$3$, a contradiction. 
If $f$ and $g$ do not meet outside of $e$, then $f,e,g$ forms a Berge path of length~$3$, a contradiction. 
If $h\neq e$ is a hyperedge in $\mathcal{H}$ such that $|h\cap e|\ge 2$, then $E(\mathcal{H})=\{e,h\}$ and $n=|V(\mathcal{H})|=r+r-|e\cap f|\le 2r-2$ since $\mathcal{H}$ is connected.
Otherwise, if $e'$ is a hyperedge in $\mathcal{H}\setminus \{e,h\}$, then $e',e,h$ forms a Berge path of length~$3$, a contradiction. 

Let $n\ge 2r-1$ and $|h\cap e|=1$ for each hyperedge $h\neq e$. Then all hyperedges in $E(\mathcal{H})\setminus \{e\}$ meet $e$ at the same vertex. 
Hence,
\[e(\mathcal{H})\le \frac{n-r}{r-1}+1=\frac{n-1}{r-1},\]
with equality holding only when $(r-1)|(n-1)$ and the extremal hypergraph is a star hypergraph which is composed of $\frac{n-1}{r-1}$ hyperedges sharing one common vertex.

Thus if $n\le 2r-2$, then $\ex^{\mathrm{conn}}_r(n,\mathcal{B}P_3)=2$ and the extremal hypergraphs consist of two hyperedges sharing at least two common vertices. 
If $n\ge 2r-1$ and $n-1$ is a multiple of $r-1$, then $\ex^{\mathrm{conn}}_r(n,\mathcal{B}P_3)=\frac{n-1}{r-1}$, and the extremal hypergraph is a star hypergraph which is composed of $\frac{n-1}{r-1}$ hyperedges sharing one common vertex. 
For the case when $n\ge 2r-1$ and $n-1$ is not a multiple of $r-1$, the function is not defined.
\end{proof}
\begin{proof}[Proof of Proposition~\ref{prop2}]
Assume that $\mathcal{H}$ is an $n$-vertex connected $\mathcal{B}P_4$-free $r$-uniform hypergraph.
If $\mathcal{H}$ is $\mathcal{B}P_3$-free, then by Proposition~\ref{prop1} we have 
\[e(\mathcal{H})\le \max\bigg\{2,\frac{n-1}{r-1}\bigg\}<\frac{n-3}{r-1}+2.\] 
Now we may assume that $\mathcal{H}$ contains a Berge path $P$ of length 3. 
Without loss of generality, we denote it by 
\[P=v_1,e_1,v_2,e_2,v_3,e_3,v_4\] 
such that $v_{i},v_{i+1}\in e_i$ for $i=1,2,3$. 
If there exists another hyperedge $f\in \mathcal{H}$ disjoint from $e_1\cup e_2\cup e_3$, then by the connectivity of $\mathcal{H}$ there must be a Berge path of length~4, a contradiction.
Thus, every hyperedge of $\mathcal{H}$ meets $e_1\cup e_2\cup e_3$.
We have $f\cap (e_1\cup e_3)\subseteq \{v_2,v_3\}$ for each $f\in E(\mathcal{H})\setminus \{e_1,e_2,e_3\}$. 
Otherwise, we can find a Berge path of length~$4$, a contradiction.
Suppose that there exist two hyperedges $f$ and $g$ in $E(\mathcal{H})\setminus \{e_1,e_2,e_3\}$ such that $f$ and $g$ meet outside of $e_1\cup e_2\cup e_3$.
Then, we can find a Berge path of length~$4$ formed by $e_1,e_2,f,g$ or $f,g,e_2,e_3$ or $e_1,f,g,e_3$, a contradiction. 
Thus, every hyperedge of $\mathcal{H}$ meets $e_1\cup e_2\cup e_3$ and there are no hyperedges which meet outside of $e_1\cup e_2\cup e_3$.

Suppose that all hyperedges in $E(\mathcal{H})\setminus \{e_1,e_2,e_3\}$ meet $e_1\cup e_2\cup e_3$ at only one vertex. Then we have 
$$
e(\mathcal{H})\le \frac{n-r-2}{r-1}+3=\frac{n-3}{r-1}+2,
$$
since $e_1\cup e_2\cup e_3$ contains at least $r+2$ vertices.  
Next we assume that there is a hyperedge $f\in E(\mathcal{H})\setminus \{e_1,e_2,e_3\}$ such that $f$ meets $e_1\cup e_2\cup e_3$ in at least two vertices. 
Since $f\cap (e_1\cup e_3)\subseteq \{v_2,v_3\}$, we have $f\cap (e_1\cup e_2\cup e_3)=f\cap e_2$. Hence, $|f\cap e_2|\ge 2$. Now we distinguish three cases.

\begin{case}
$\{v_2,v_3\}\subset f\cap e_2$.    
\end{case}
In this case, any hyperedge in $E'=E(\mathcal{H})\setminus \{e_1,e_2,e_3,f\}$ can contain only the vertices $v_2$ and $v_3$. 
Otherwise, if $g\in E'$ satisfies that $a\in g\cap e_2$ for some vertex $a\neq v_2,v_3$ in $e_2$, we can find a Berge path of length~$4$ formed by $g,e_2,f,e_1$. 
If $|f\cap e_2|\ge 3$, then $E(\mathcal{H})=\{e_1,e_2,e_3,f\}$. 
Otherwise, if $g\in E(\mathcal{H})\setminus \{e_1,e_2,e_3,f\}$, then $g,f,e_2,e_3$ or $g,f,e_2,e_1$ forms a Berge path of length 4, a contradiction. Assume $|f\cap e_2|=2$. 
Then, we have 
\[
e(\mathcal{H})\le \frac{n-r-2}{r-2}+3=\frac{n-4}{r-2}+2,
\]
since $e_1\cup e_2\cup e_3$ contains at least $r+2$ vertices.
\begin{case}
Either $v_2$ or $v_3$ is in $f\cap e_2$.
\end{case}
Without loss of generality, we assume $v_2\in f\cap e_2$ and $v_3\notin f\cap e_2$. 
Since $|f\cap e_2|\ge 2$, we have a Berge path of length $4$ formed by $e_1,f,e_2,e_3$, a contradiction.
\begin{case}
$v_2,v_3\notin f$.
\end{case}
If there exist hyperedges $f,g\in E(\mathcal{H})\setminus \{e_1,e_2,e_3\}$ such that $|f\cap e_2|\ge 2$ and $|e\cap f\cap g|\ge 1$, then $g,f,e_2,e_3$ forms a Berge path of length~$4$, a contradiction.
Assume that $h\in E(\mathcal{H})\setminus \{e_1,e_2,e_3,f\}$ and $|h\cap e_2|\ge 2$. 
If $\{v_2,v_3\}\in h\cap e_2$, then $f,e_2,g,e_3$ forms a Berge path of length~4, a contradiction. 
If $v_2\in h\cap e_2$ and $v_3\notin h\cap e_2$, then $e_1,g,e_2,f$ forms a Berge path of length~4, a contradiction. 
Similarly, if $v_3\in h\cap e_2$ and $v_2\notin h\cap e_2$, then $e_3,g,e_2,f$ forms a Berge path of length~4, a contradiction. 
Hence, any two hyperedges in $E(\mathcal{H})\setminus \{e_1,e_2,e_3\}$ are disjoint. 
Assume that there are $m$ hyperedges $f_1,f_2,\ldots,f_m$ such that $|f_i\cap e_2|\ge 2$. 
Then we have $e(\mathcal{H})-m-3$ hyperedges meeting $e_2$ at only one vertex. 
Note that $\sum_{i=1}^m |f_i\cap e_2|\le r-2$ and $1\le m\le \frac{r-2}{2}$. Now we consider the number of vertices spanned by all hyperedges of $\mathcal{H}$.
Firstly, the Berge path $P$ spans at least $r+2$ vertices. 
All hyperedges in $\{f_1,f_2,\ldots,f_m\}$ span at least $mr-r+2$ new vertices.
All hyperedges in $E(\mathcal{H})\setminus \{e_1,e_2,e_3,f_1,f_2,\ldots,f_m\}$ span $(r-1)(e(\mathcal{H})-3-m)$ vertices. 
Hence,
\[
n\ge r+2+mr-r+2+(r-1)(e(\mathcal{H})-3-m)
\]
which implies that
\[
e(\mathcal{H})\le \frac{n-mr-4}{r-1}+m+3.
\]
Define 
\[
f(m)=\frac{n-mr-4}{r-1}+m+3.
\]
The first-order derivative of $f(m)$ is
\[
f'(m)=1-\frac{r}{r-1}=-\frac{1}{r-1}<0.
\]
Hence, $f(m)$ is a monotonically decreasing function in $m$.
Since $1\le m\le \frac{r-2}{2}$, 
\[
e(\mathcal{H})\le f(m)\le f(1)=\frac{n-r-4}{r-1}+4=\frac{n-5}{r-1}+3.
\]

Combining the results of the above discussion, we obtain that 
\[
\ex^{\mathrm{conn}}_r(n,\mathcal{B}P_4)\le \max\bigg\{4,\frac{n-5}{r-1}+3,\frac{n-4}{r-2}+2\bigg\}.
\]
Note that $\ex^{\mathrm{conn}}_r(n,\mathcal{B}P_4)\le 4$ when $n\le r+4$. 
\[
\ex^{\mathrm{conn}}_r(n,\mathcal{B}P_4)\le \max\bigg\{4,\frac{n-5}{r-1}+3,\frac{n-4}{r-2}+2\bigg\}
\]
when $n\ge r+5$.

To show the lower bound, we need to construct extremal hypergraphs.
Let $A=\{u_1,u_2,\ldots,\linebreak u_{r-2}\}$ be a set of vertices, and let $v_1,v_2,v_3,v_4$ be four distinct other vertices.
Consider the three hyperedges $e_1,e_2,e_3$ with $e_1=\{v_1,v_2\}\cup A$, $e_2=\{v_2,v_3\}\cup A$ and $e_3=\{v_3,v_4\}\cup A$.
For the case $n\le r+4$, we let $\mathcal{H}$ contain $e_1,e_2,e_3$ and one other hyperedge which meets $e_2$ in at most 3 vertices.
For the other case $n\ge r+5$, we consider the following constructions.
Assume that $n-4$ is a multiple of $r-2$. Let $\mathcal{H}_1$ be an $n$-vertex $r$-uniform hypergraph such that $\{e_1,e_2,e_3\}\subset E(\mathcal{H}_1)$ and $A\cup \{v_1,v_2,v_3,v_4\}\subset V(\mathcal{H}_1)$. 
The remaining hyperedges in $\mathcal{H}_1$ satisfy that they share two common vertices $v_2,v_3$ and they meet $e_1\cup e_2\cup e_3$ only at $v_2,v_3$.
It is easy to verify that $\mathcal{H}_1$ is connected and $\mathcal{B}P_4$-free. 
Assume that $n-5$ is a multiple of $r-1$. 
Let $\mathcal{H}_2$ be an $n$-vertex $r$-uniform hypergraph such that $\{e_1,e_2,e_3\}\subset E(\mathcal{H}_2)$ and $A\cup \{v_1,v_2,v_3,v_4\}\subset V(\mathcal{H}_2)$. 
Add 
a hyperedge $f$ with $f=A\cup \{v_5,v_6\}$.
The remaining hyperedges in $\mathcal{H}_2$ satisfy that they share one common vertex $v_3$, they meet $e_1\cup e_3$ only at $v_3$ and they are disjoint from $f$. 
It is easy to verify that $\mathcal{H}_2$ is connected and $\mathcal{B}P_4$-free. 

This completes the proof.
\end{proof}

\begin{proof}[Proof of Theorem \ref{thm1}]
For the lower bound on $\ex^{\mathrm{conn}}_r(n,\mathcal{B}P_k)$, we construct the extremal $r$-uniform hypergraph $\mathcal{H}$ as follows. 
Let $n=1+ar+b$ with $a\ge 0$ and $0\le b<r$. 
We consider $a-1$ copies of an $r$-uniform hypergraph $\mathcal{H}_1$ and an $r$-uniform hypergraph $\mathcal{H}_2$ such that all $a-1$ copies of $\mathcal{H}_1$ and $\mathcal{H}_2$ share one common vertex, where $\mathcal{H}_1$ has $\lfloor\frac{k-1}{2}\rfloor$ hyperedges and $r+1$ vertices and  $\mathcal{H}_2$ has $\lceil\frac{k-1}{2}\rceil$ hyperedges and the remaining $r+b$ vertices. 
It is easy to verify that $\mathcal{H}$ is $\mathcal{B}P_k$-free and connected. Hence, we have $\ex^{\mathrm{conn}}_r(n,\mathcal{B}P_k)\ge \lfloor\frac{k-1}{2}\rfloor\lfloor\frac{n-1}{r}\rfloor+\mathbbm{1}_{2|k}$.

To show the upper bound on $\ex^{\mathrm{conn}}_r(n,\mathcal{B}P_k)$, we let $\mathcal{H}$ be an $n$-vertex connected $\mathcal{B}P_k$-free $r$-uniform hypergraph.
Suppose that a longest Berge path in $\mathcal{H}$ has length $t$.
Note that $t\le k-1$. 
Assume that $\mathcal{B}C_t$ is a Berge cycle of length $t$ in $\mathcal{H}$. 
Denote by $U=\{v_1,v_2,\ldots,v_t\}$ and $\mathcal{F}=\{e_1,e_2,\ldots,e_t\}$, the defining vertices and hyperedges of this cycle. 
This means $\{v_i,v_{i+1}\}\subset e_i$ for $1\le i\le t-1$ and $\{v_1,v_t\}\subset e_t$. 

Note that there exists no hyperedge in $\mathcal{H}[V(\mathcal{H})\setminus U]$.
Otherwise, by the connectivity we can extend a Berge path of length $t-1$ in $\mathcal{B}C_t$ to a Berge path of length $t+1$, a contradiction. 
Furthermore, there exists no Berge path of length~2 such that one terminal defining vertex is from $U$ and the other defining vertices are in $V(\mathcal{H}\setminus U)$. 
Otherwise, such a Berge path $P_A$ of length 2 and the Berge path $P_B$ of length $t-1$ in $\mathcal{B}C_t$ satisfying one terminal defining vertex of $P_B$ is the same as one terminal defining vertex of $P_A$ together constitute a Berge path of length $t+1$, a contradiction. 
There cannot be two hyperedges $e,f$ in $E(\mathcal{H})\setminus \mathcal{F}$ such that $v_i\in e\cap U$ and $v_{i+1}\in f\cap U$. 
Otherwise, we can find a Berge path of length $t+1$ formed by $e,\mathcal{B}C_t,f$, a contradiction. 
For any two distinct vertices $v_i,v_j\in U$, there cannot be two hyperedges $e,f$ such that $v_i\in e\cap U, v_j\in f\cap U$ and $(e\setminus U)\cap (f\setminus U)\neq \emptyset$.
If there is a hyperedge $f\in E(\mathcal{H})\setminus \mathcal{F}$ such that $|f\cap U|>\lfloor\frac{t}{2}\rfloor$, then $E(\mathcal{H})=\mathcal{F}\cup \{f\}$. Now we may assume that $|f\cap U|\le \lfloor\frac{t}{2}\rfloor$ for all hyperedges in $E(\mathcal{H})\setminus \mathcal{F}$.
Hence, we have 
\[e(\mathcal{H})\le t+\frac{n-t}{r-\left\lfloor\frac{t}{2}\right\rfloor}.
\]
Define
\[
f(t)=t+\frac{n-t}{r-\left\lfloor\frac{t}{2}\right\rfloor}.
\]
By calculating the first derivative of $f(t)$, we get
\[
f'(t)=1+\frac{2n-4r-2}{(2r-t+1)^2}
\]
when $t$ is odd and
\[
f'(t)=1+\frac{2n-4r}{(2r-t)^2}
\]
when $t$ is even.
Note that $f'(t)>0$ whenever $t$ is odd or even since $r\ge k>t$. 
Hence, $f(t)$ is a monotonically increasing in $t$. Since $t\le k-1$,
\[
e(\mathcal{H})\le f(k-1)=k-1+\frac{n-k+1}{r-\left\lfloor\frac{k-1}{2}\right\rfloor}.
\]
We can construct an extremal hypergraph $\mathcal{H}'$ as follows. 
Consider a Berge cycle of length $k-1$ with defining vertices $v_1,v_2,\ldots,v_{k-1}$ and defining hyperedge $e_1,e_2,\ldots,e_{k-1}$. 
Define all other hyperedges such that they share the common nonadjacent $\lfloor\frac{k-1}{2}\rfloor$ vertices in $U=\{v_1,v_2,\ldots,v_{k-1}\}$. 
It is easy to verify that $\mathcal{H}'$ is connected and $\mathcal{B}P_k$-free.

Next we may assume that $\mathcal{H}$ contains no Berge cycle of length $t$. Since a longest Berge path in $\mathcal{H}$ has length $t$, we have $\mathcal{H}$ is $\mathcal{BC}_{\ge t}$-free. 
We prove Theorem~\ref{thm1} by induction on $n$. Assume that $n=2r+1$.
Note that by Lemma~\ref{lemma20} we have
$e(\mathcal{H})= |N_{\mathcal{H}}(S)|+|E(\mathcal{H[V(\mathcal{H})\setminus S]})|$, where $S$ is a vertex set of size at least $2r-2$ such that $|N_{\mathcal{H}}(S)|\le 2$ or a vertex set of size at least $2r-1$ such that $|N_{\mathcal{H}}(S)|\le t$.
Then 
\begin{equation*}
\begin{aligned}
e(\mathcal{H})&= |N_{\mathcal{H}}(S)|+|E(\mathcal{H[V(\mathcal{H})\setminus S]})|\\
&\le t\\
&\le k-1\\
&\le \left\lfloor\frac{k-1}{2}\right\rfloor\left\lfloor\frac{n-1}{r}\right\rfloor+\mathbbm{1}_{2|k}.
\end{aligned}
\end{equation*}

Now we suppose $n\ge 2r+2$ and for any $\mathcal{B}P_k$-free connected $r$-uniform hypergraph $\mathcal{H}$ with $n'<n$ ($n'$ is sufficiently large in terms of $k$ and $r$) vertices we have 
\begin{equation*}
e(\mathcal{H})\le \left\lfloor\frac{k-1}{2}\right\rfloor\left\lfloor\frac{n'-1}{r}\right\rfloor+\mathbbm{1}_{2|k}.
\end{equation*}
Next we will show it holds for $n$. 
Suppose that $S$ is a vertex set of size at least $2r-2$ such that $|N_{\mathcal{H}}(S)|\le 2$ or a vertex set of size at least $2r-1$ such that $|N_{\mathcal{H}}(S)|\le t$.
Let $\mathcal{H}'$ be the hypergraph induced by $V'=V(\mathcal{H})\setminus S$. 

Assume that $\mathcal{H}'$ has $m$ connected components $A_1,A_2,\ldots,A_m$ ($m\ge 1$), the number of vertices in $A_i$ is $a_i$ and a longest Berge path in $A_i$ has length $s_i$ for each $1\le i\le m$.  
If each $A_i$ is $\mathcal{B}C_{\ge s_i}$-free for each $1\le i\le m$, then when $n$ is not a multiple of $r$ we have either 
\begin{equation*}
\begin{aligned}
e(\mathcal{H})&=|N_{\mathcal{H}}(S)|+\sum_{i=1}^m |E(A_i)|\\
&\le 2+\sum_{i=1}^m \bigg(\left\lfloor\frac{s_i-1}{2}\right\rfloor\left\lfloor\frac{a_i-1}{r}\right\rfloor+\mathbbm{1}_{2|s_i}\bigg)\\
&\le 2+\left\lfloor\frac{k-1}{2}\right\rfloor\left\lfloor\frac{n-1-(2r-2)}{r}\right\rfloor+\mathbbm{1}_{2|k}\\
&\le \left\lfloor\frac{k-1}{2}\right\rfloor\left\lfloor\frac{n-1}{r}\right\rfloor+\mathbbm{1}_{2|k}
\end{aligned}
\end{equation*}
or
\begin{equation*}
\begin{aligned}
e(\mathcal{H})&=|N_{\mathcal{H}}(S)|+\sum_{i=1}^m |E(A_i)|\\
&\le k-1+\sum_{i=1}^m \bigg(\left\lfloor\frac{s_i-1}{2}\right\rfloor\left\lfloor\frac{a_i-1}{r}\right\rfloor+\mathbbm{1}_{2|s_i}\bigg)\\
&\le k-1+\left\lfloor\frac{k-1}{2}\right\rfloor\left\lfloor\frac{n-1-(2r-1)}{r}\right\rfloor+\mathbbm{1}_{2|k}\\
&\le \left\lfloor\frac{k-1}{2}\right\rfloor\left\lfloor\frac{n-1}{r}\right\rfloor+\mathbbm{1}_{2|k}.
\end{aligned}
\end{equation*}

Assume that there are $m_1$ connected components $A_1,A_2,\ldots, A_{m_1}$ in $\mathcal{H}'$ such that $A_i$ is not $\mathcal{BC}_{s_i}$-free for each $1\le i\le m_1$. 
Then, when $n$ is not a multiple of $r$ and $n$ is sufficiently large, since $m_1\le (r-1)(k-1)$,
we have either 
\begin{equation*}
\begin{aligned}
e(\mathcal{H})&=|N_{\mathcal{H}}(S)|+\sum_{i=1}^{m_1} |E(A_i)|+\sum_{i=m_1}^{m} |E(A_i)|\\
&\le 2+\sum_{i=1}^{m_1} \bigg(\max\bigg\{s_i+1,s_i+\frac{a_i-s_i}{r-\lfloor\frac{s_i}{2}\rfloor}\bigg\}\bigg)+\sum_{i=m_1}^m \bigg(\left\lfloor\frac{s_i-1}{2}\right\rfloor\left\lfloor\frac{a_i-1}{r}\right\rfloor+\mathbbm{1}_{2|s_i}\bigg)\\
&\le 2+\left\lfloor\frac{k-1}{2}\right\rfloor\left\lfloor\frac{n-1-(2r-2)}{r}\right\rfloor+\mathbbm{1}_{2|k}\\
&\le \left\lfloor\frac{k-1}{2}\right\rfloor\left\lfloor\frac{n-1}{r}\right\rfloor+\mathbbm{1}_{2|k}
\end{aligned}
\end{equation*}
or
\begin{equation*}
\begin{aligned}
e(\mathcal{H})&=|N_{\mathcal{H}}(S)|+\sum_{i=1}^{m_1} |E(A_i)|+\sum_{i=m_1}^{m} |E(A_i)|\\
&\le k-1+\sum_{i=1}^{m_1} \bigg(\max\bigg\{s_i+1,s_i+\frac{a_i-s_i}{r-\lfloor\frac{s_i}{2}\rfloor}\bigg\}\bigg)+\sum_{i=m_1}^m \bigg(\left\lfloor\frac{s_i-1}{2}\right\rfloor\left\lfloor\frac{a_i-1}{r}\right\rfloor+\mathbbm{1}_{2|s_i}\bigg)\\
&\le k-1+\left\lfloor\frac{k-1}{2}\right\rfloor\left\lfloor\frac{n-1-(2r-1)}{r}\right\rfloor+\mathbbm{1}_{2|k}\\
&\le \left\lfloor\frac{k-1}{2}\right\rfloor\left\lfloor\frac{n-1}{r}\right\rfloor+\mathbbm{1}_{2|k}.
\end{aligned}
\end{equation*}

This completes the proof.
\end{proof}

\begin{proof}[Proof of Lemma \ref{thm3}]
For the lower bound, we construct the extremal hypergraphs as follows. 
When $r\le n\le 2r-2$, the extremal hypergraphs are $n$-vertex $r$-uniform hypergraphs with $k-1$ hyperedges.
Let $n\ge 2r-1$, and suppose $n$ is not a multiple of $r-1$. 
One extremal hypergraph $\mathcal{H}_1$ is constructed in the following way. 
Let $n=1+a(r-1)+b$ with $a\ge 2$ and $0\le b\le r-1$. 
Consider $a-1$ hyperedges with multiplicity $\lfloor\frac{k-1}{2}\rfloor$ which share one common vertex $v$. 
And consider another hypergraph $\mathcal{H}'_1$ with $\lceil\frac{k-1}{2}\rceil$ hyperedges and the remaining $r-1+b$ vertices. 
Choose this $\mathcal{H}'_1$ in such a way that it shares one common vertex $v$ with the above $a-1$ hyperedges. 
We can construct another extremal hypergraph $\mathcal{H}_2$ as follows. 
Start with a Berge cycle of length $k-1$ with defining vertices $v_1,v_2,\ldots,v_{k-1}$ and defining hyperedge $e_1,e_2,\ldots,e_{k-1}$ which correspond to one hyperedge with multiplicity $k-1$. 
Choose all other hyperedges in such a way that they share the common non-adjacent $\lfloor\frac{k-1}{2}\rfloor$ vertices in $U=\{v_1,v_2,\ldots,v_{k-1}\}$. 

For the upper bound, we let $\mathcal{H}$ be a connected $n$-vertex $\mathcal{B}P_k$-free $r$-uniform multi-hypergraph.
We assume that each hyperedge of $\mathcal{H}$ has multiplicity at most $m$. 
Note that $m\le k-1$.
Otherwise, if $\mathcal{H}$ contains a hyperedge $e$ which has multiplicity $\ge k$, then there is a Berge cycle $C$ of length $k$ in $\mathcal{H}$. 
By the connectivity of $\mathcal{H}$, we can find a Berge path of length $k$, a contradiction. Hence, $m\le k-1$.

Assume that $\mathcal{H}$ contains a hyperedge $e'$ which has multiplicity $k-1$. 
Then, there is a Berge cycle $C'$ of length $k-1$ in $\mathcal{H}$. Denote by $v_1,v_2,\ldots,v_{k-1}$ the defining vertices of $C'$ and $e_1,e_2,\ldots,e_{k-1}$ the defining hyperedges of $C'$ such that for any $1\le i\le k-2$, we have $v_i,v_{i+1}\in e_i$ and $v_{k-1},v_1\in e_{k-1}$.
Note that there is no hyperedge in $\mathcal{H}\setminus \{v_1,v_2,\ldots,v_{k-1}\}$. 
Otherwise, by the connectivity of $\mathcal{H}$ we can find a Berge path of length $k$, a contradiction.
Also, any other hyperedges intersecting $C'$ have multiplicity exactly one. 
Moreover, there exists no Berge path of length 2 with one terminal defining vertex from $\{v_1,v_2,\ldots,v_{k-1}\}$ and the other defining vertices in $V(\mathcal{H})\setminus \{v_1,v_2,\ldots,v_{k-1}\}$. 
Otherwise, such a Berge path $P_A$ of length 2 and the Berge path $P_B$ of length $k-2$ in $C'$ satisfying that one terminal defining vertex of $P_B$ is the same as one terminal defining vertex of $P_A$ together constitute a Berge path of length $k$, a contradiction. 
There cannot be two hyperedges $e,f$ in $\mathcal{H}$ such that $v_i\in e\cap \{v_1,v_2,\ldots,v_{k-1}\}$ and $v_{i+1}\in f\cap \{v_1,v_2,\ldots,v_{k-1}\}$. 
Otherwise, we can find a Berge path of length $k$ in $e,f,C$ or $e,C,f$, a contradiction. For any two distinct vertices $v_i,v_j\in \{v_1,v_2,\ldots,v_{k-1}\}$, there cannot be two hyperedges $e,f$ such that $v_i\in e\cap \{v_1,v_2,\ldots,v_{k-1}\}, v_j\in f\cap \{v_1,v_2,\ldots,v_{k-1}\}$ and $(e\setminus \{v_1,v_2,\ldots,v_{k-1}\})\cap (f\setminus \{v_1,v_2,\ldots,v_{k-1}\})\neq \emptyset$. 
Hence, we have 
$$
e(\mathcal{H})\le k-1+\frac{n-k+1}{r-\left\lfloor\frac{k-1}{2}\right\rfloor}.
$$

Next we assume that $m\le k-2$. Suppose that a longest Berge path in $\mathcal{H}$ has length $t$.
Note that $t\le k-1$. 
Firstly, we assume that $\mathcal{H}$ contains a Berge cycle $C$ of length $t$. Denote by $U=\{v_1,v_2,\ldots,v_t\}$ and $\mathcal{F}=\{e_1,e_2,\ldots,e_t\}$, the defining vertices and hyperedges of $C$ such that for $1\le i\le t-1$ we have $v_i,v_{i+1}\in e_i$ and $v_1,v_t\in e_t$. 
Note that each $e_i$ has multiplicity one for any $1\le i\le t$. Otherwise, we can find a Berge path of length $t+1$, a contradiction.
Note that there exists no hyperedge in $\mathcal{H}[V(\mathcal{H})\setminus U]$. Otherwise, by the connectivity we can extend a Berge path of length $t-1$ in $C$ to a Berge path of length $t+1$, a contradiction.
Also, any other hyperedges intersecting $C$ have multiplicity exactly one. Moreover, there exists no Berge path of length 2 with one terminal defining vertex from $U$ and the other defining vertices in $V(\mathcal{H}\setminus U)$. Otherwise, such a Berge path $P_A$ of length 2 and the Berge path $P_B$ of length $t-1$ in $C$ satisfying that one terminal defining vertex of $P_B$ is the same as one terminal defining vertex of $P_A$ together constitute a Berge path of length $t+1$, a contradiction. 
There cannot be two hyperedges $e,f$ in $\mathcal{H}$ such that $v_i\in e\cap U$ and $v_{i+1}\in f\cap U$. Otherwise, we can find a Berge path of length $t+1$ in $e,f,\mathcal{B}C_t$ or $e,\mathcal{B}C_t,f$, a contradiction. For any two distinct vertices $v_i,v_j\in U$, there cannot be two hyperedges $e,f$ such that $v_i\in e\cap U, v_j\in f\cap U$ and $(e\setminus U)\cap (f\setminus U)\neq \emptyset$. 
Hence, we have 
$$
e(\mathcal{H})\le t+\frac{n-t}{r-\left\lfloor\frac{t}{2}\right\rfloor}\le k-1+\frac{n-k+1}{r-\left\lfloor\frac{k-1}{2}\right\rfloor}.
$$

Now we assume that $\mathcal{H}$ contains no Berge cycle of length $t$. We prove the theorem by induction on the number of vertices. 
The theorem trivially holds for $n\le r-1$.  When $r\le n\le 2r-2$, by Theorem \ref{multi} and the assumption that $\mathcal{H}$ contains no Berge cycle of length $t$, we have that $e(\mathcal{H})\le t-1<k-1$.

For the case $n=2r+1$, by Lemma~\ref{lemma20} we have that either there exists a vertex subset $S$ of size at least $2r-2$ such that $|N_{\mathcal{H}}(S)|\le m+1\le k-1$ or there exists a vertex subset $S$ of size at least $2r-1$ such that $|N_{\mathcal{H}}(S)|\le k-1$. Hence, 
$e(\mathcal{H})\le k-1=\lfloor\frac{n-1}{r-1}\rfloor\lfloor\frac{k-1}{2}\rfloor+\mathbbm{1}_{2|k}$.

Now we suppose $n>2r-1$ and that the theorem holds for any connected $r$-uniform multi-hypergraph with $n'<n$ ($n'$ is not multiple of $r-1$) vertices.
By Lemma~\ref{lemma20}, we have that either there exists a vertex subset $S$ of size at least $2r-2$ such that $|N_{\mathcal{H}}(S)|\le m+1\le k-1$ or there exists a vertex subset $S$ of size at least $2r-1$ such that $|N_{\mathcal{H}}(S)|\le k-1$.  Let $\mathcal{H}'$ be the hypergraph induced by $V'=V(\mathcal{H})\setminus S$.
Then 
\begin{equation*}
\begin{aligned}
e(\mathcal{H})&\le (k-1)+e(\mathcal{H'})\\
&\le (k-1)+\left\lfloor\frac{n-1-(2r-2)}{r-1}\right\rfloor\left\lfloor\frac{k-1}{2}\right\rfloor+\mathbbm{1}_{2|k}\\
&\le \left\lfloor\frac{n-1}{r-1}\right\rfloor\left\lfloor\frac{k-1}{2}\right\rfloor+\mathbbm{1}_{2|k}
\end{aligned}
\end{equation*}
when $n$ is not a multiple of $r-1$.

Hence, we have $e(\mathcal{H})=0$ if $n\le r-1$, $e(\mathcal{H})\le k-1$ if $r\le n\le 2r-2$ and 
\[
e(\mathcal{H})\le \left\lfloor\frac{n-1}{r-1}\right\rfloor\left\lfloor\frac{k-1}{2}\right\rfloor+\mathbbm{1}_{2|k}
\]
if $n\ge 2r-1$ and $n$ is not a multiple of $r-1$.

This completes the proof.    
\end{proof}

\begin{proof}[Proof of Theorem \ref{thm2}]
For the lower bound, we construct an extremal hypergraph $\mathcal{H}$ as follows. 
The vertex set of $\mathcal{H}$ is $\{v_1,v_2,\ldots,v_{n}\}$ and the edge set is $\{\{v_1,v_2,\ldots,v_{r-1},v_i\}:r\le i\le n\}$.
It is easy to verify that $\mathcal{H}$ is connected and $\mathcal{B}P_{r+1}$-free.

For the upper bound, we let $\mathcal{H}$ be a connected $n$-vertex $r$-uniform hypergraph which is $\mathcal{B}P_{r+1}$-free. 
We discuss it into two cases. At first, if $\mathcal{H}$ is also $\mathcal{BC}_{\ge r}$-free, then by Theorem~\ref{thmcycle} we have that $e(\mathcal{H})\le n-r+1$ when $n$ is large. 
Next we may assume that $\mathcal{H}$ contains a Berge cycle of length $r$. Indeed, if $\mathcal{H}$ contains a Berge cycle of length  at least $r+2$, we can indirectly find a Berge path of length $r+1$.
If $\mathcal{H}$ contains a Berge cycle of length $r+1$, then we can still find a Berge path of length $r+1$ since $\mathcal{H}$ is connected and $n$ is large. 

Denote by $\mathcal{B}C_r$ a Berge cycle of length $r$ in $\mathcal{H}$, by $v_1,v_2,\ldots,v_{r}$ the defining vertices, and by $e_1,e_2,\ldots,e_r$ the defining hyperedges such that for any $1\le i\le r$ we have $v_i,v_{i+1}\in e_i$ and $v_r,v_1\in e_r$.
Note that there exists no hyperedge in $\mathcal{H}[V(\mathcal{H})\setminus \{v_1,v_2,\ldots,v_{r}\}]$. Otherwise, by the connectivity we can extend a Berge path of length $r-1$ in $\mathcal{B}C_{r}$ to a Berge path of length $r+1$, a contradiction. And there exists no Berge path of length 2 with one terminal defining vertex from $\{v_1,v_2,\ldots,v_{r}\}$ and the other defining vertices in $V(\mathcal{H}\setminus \{v_1,v_2,\ldots,v_{r}\})$. 
Otherwise, such a Berge path $P_A$ of length 2 and the Berge path $P_B$ of length $r-1$ in $\mathcal{B}C_r$ satisfying that one terminal defining vertex of $P_B$ is the same as one terminal defining vertex of $P_A$ together constitute a Berge path of length $r+1$, a contradiction. 
There cannot be two hyperedges $e,f$ in $\mathcal{H}$ such that $v_i\in e\cap \{v_1,v_2,\ldots,v_{r}\}$ and $v_{i+1}\in f\cap \{v_1,v_2,\ldots,v_{r}\}$. 
Otherwise, we can find a Berge path of length $r+1$ in $e,f,\mathcal{B}C_r$ or $e,\mathcal{B}C_r,f$, a contradiction. For any two distinct vertices $v_i,v_j\in \{v_1,v_2,\ldots,v_{r}\}$, there cannot be two hyperedges $e,f$ such that $v_i\in e\cap \{v_1,v_2,\ldots,v_{r}\}, v_j\in f\cap \{v_1,v_2,\ldots,v_{r}\}$ and $(e\setminus \{v_1,v_2,\ldots,v_{r}\})\cap (f\setminus \{v_1,v_2,\ldots,v_{r}\})\neq \emptyset$. 
Hence, we have 
$$
e(\mathcal{H})\le r+\frac{n-r}{r-\left\lfloor\frac{r}{2}\right\rfloor}\le n-r+1
$$
when $n$ is large. 
This completes the proof.
\end{proof}

\section{Proof of Lemma \ref{lemma20}}\label{sec4}
Fix integers $n>r\ge k\ge 3$ and let $\mathcal{H}$ be an $n$-vertex connected $\mathcal{B}P_k$-free $r$-uniform multi-hypergraph. 
Consider a longest Berge path 
$$
P=e_1,v_1,e_2,v_2,\ldots,e_{t-1},v_{t-1},e_t
$$
such that $v_1\in e_1, v_{t-1}\in e_t$ and $\{v_{i-1},v_i\}\subset e_i$ for $i=2,3,\ldots,t-1$. Let $\mathcal{F}=E(P)=\{e_1,e_2,\ldots,e_t\}$ and $U=\{v_1,v_2,\ldots,v_{t-1}\}$, the defining hyperedges and vertices of this path.

\begin{claim}
$t\le k-1$.    
\end{claim}
\begin{proof}
If $t\ge k+2$, then $v_1,e_2,v_2,e_3,\ldots,e_{k+1},v_{k+1}$ is a Berge path of length $k$, a contradiction. Hence, $t\le k+1$.
Suppose that $t=k+1$. 
Then, $P=e_1,v_1,e_2,v_2,\ldots,e_k,v_k,e_{k+1}$.
If $k=r$ and $V(e_1)=V(e_{k+1})=\{v_1,v_2,\ldots,v_k\}$, then we have a Berge cycle of length $k$. 
By the connectivity of $\mathcal{H}$, we can extend it to a Berge path of length $k$, a contradiction.
If either $V(e_1)\neq \{v_1,v_2,\ldots,v_k\}$ or $V(e_{k+1})\neq \{v_1,v_2,\ldots,v_k\}$ holds, we can find a Berge path of length $k$, a contradiction. Hence $t\le k$.

Suppose that $t=k$. Then $P=e_1,v_1,e_2,v_2,\ldots,e_{k-1},v_{k-1},e_k$. Since $r\ge k>k-1$, we have $e_1\setminus \{v_1,v_2,\ldots,v_{k-1}\}\neq \emptyset$ and $e_k\setminus \{v_1,v_2,\ldots,v_{k-1}\}\neq \emptyset$. Since there is no Berge cycle of length $k$ in $\mathcal{H}$,  we have $(e_1\cap e_t)\setminus \{v_1,v_2,\ldots,v_{k-1}\}=\emptyset$. 
Then, we can find a Berge path of length $k$, a contradiction.
Thus, we have $t\le k-1$.
\end{proof}

Note that $e_1\setminus U\neq \emptyset$ and $e_t\setminus U\neq \emptyset$.
\begin{lemma}\label{lemma21}
Suppose $w_1\in e_1\setminus U$ and $w_2\in e_t\setminus U$. Then $N_{\mathcal{H}}(w_1)\subseteq \mathcal{F}\setminus \{e_t\}$ and $N_{\mathcal{H}}(w_2)\subseteq \mathcal{F}\setminus \{e_1\}$. Hence $N_{\mathcal{H}}(e_1\setminus U)\subseteq \mathcal{F}\setminus \{e_t\}$ and $N_{\mathcal{H}}(e_t\setminus U)\subseteq \mathcal{F}\setminus \{e_1\}$.
\end{lemma}
\begin{proof}
Suppose that there exists a hyperedge $f\notin \mathcal{F}\setminus \{e_t\}$ containing $w_1$ (note here $f\neq e_t$ since $\mathcal{H}$ is $\mathcal{B}{C}_t$-free). 
Then, $f,w_1,P$ is a longer Berge path, a contradiction to the maximality of $P$. 
Suppose that there exists a hyperedge $g\notin \mathcal{F}\setminus \{e_1\}$ containing $w_2$ (note here $g\neq e_1$ since $\mathcal{H}$ is $\mathcal{B}C_t$-free). 
Then, $P,w_2,g$ is a longer Berge path, a contradiction to the maximality of $P$. 
This completes the lemma.
\end{proof}

\begin{lemma}\label{lemma22}
If for some $1\le i\le t-1$ we have $v_i\in e_1\cap U$, then $N_{\mathcal{H}}(e_i\setminus U)\subseteq \mathcal{F}\setminus \{e_t\}$. If for some $1\le j\le t-1$ we have $v_{j-1}\in e_t\cap U$, then $N_{\mathcal{H}}(e_j\setminus U)\subseteq \mathcal{F}\setminus \{e_1\}$.
\end{lemma}
\begin{proof}
Consider the Berge path 
$$
e_i,v_{i-1},e_{i-1},v_{i-2},\ldots,e_2,v_1,e_1,v_i,e_{i+1},v_{i+1},\ldots,e_{t-1},v_{t-1},e_t.
$$ 
This path has length $t$ and starts at edge $e_i$. By Lemma~\ref{lemma21} we have that $N_{\mathcal{H}}(e_i\setminus U)\subseteq \mathcal{F}\setminus \{e_t\}$.

For the second statement, we consider the Berge path 
$$
e_j,v_j,e_{j+1},v_{j+1},\ldots,e_{t-1},v_{t-1},e_t,v_{j-1},e_{j-1},v_{j-2},\ldots,e_2,v_1,e_1.
$$
This path has length $t$ and starts at edge $e_j$. 
By Lemma~\ref{lemma21} we have that $N_{\mathcal{H}}(e_j\setminus U)\subseteq \mathcal{F}\setminus \{e_1\}$.
\end{proof}

\begin{lemma}\label{lemma23}
If there are two vertices $v_i,v_j\in e_1\cap U$, with $i>j$ such that $(e_i\cap e_j)\setminus U\neq \emptyset$, then $N_{\mathcal{H}}(v_{i-1})\subseteq \mathcal{F}\setminus \{e_t\}$ and $N_{\mathcal{H}}(v_j)\subseteq \mathcal{F}\setminus \{e_t\}$. 
If there are two vertices $v_{i-1},v_{j-1}\in e_t\cap U$, with $i>j$ such that $(e_i\cap e_j)\setminus U\neq \emptyset$, then $N_{\mathcal{H}}(v_{i-1})\subseteq \mathcal{F}\setminus \{e_1\}$ and $N_{\mathcal{H}}(v_j)\subseteq \mathcal{F}\setminus \{e_1\}$.
\end{lemma}
\begin{proof}
For the first statement, we let $u\in (e_i\cap e_j)\setminus U$.
We consider the following Berge paths.
\begin{equation*}
\begin{aligned}
P_A=&e_{i-1},v_{i-2},e_{i-2},v_{i-3},\ldots, e_{j+1},v_j,e_1,v_1,e_2,v_2,\ldots,\\
&v_{j-1},e_j,u,e_i,v_i,e_{i+1},v_{i+1},\ldots,e_{t-1},v_{t-1},e_t. \\   
\end{aligned}
\end{equation*}
\begin{equation*}
\begin{aligned}
P_B=&e_{j+1},v_{j+1},e_{j+2},v_{j+2},\ldots, v_{i-1},e_i,u,e_j,v_{j-1},e_{j-1},v_{j-2},\ldots,\\
&e_2,v_1,e_1,v_i,e_{i+1},v_{i+1},\ldots,e_{t-1},v_{t-1},e_t. 
\end{aligned}
\end{equation*}
Applying Lemma~\ref{lemma21} for $P_A$ and $P_B$, we have that $N_{\mathcal{H}}(v_{i-1})\subseteq \mathcal{F}\setminus \{e_t\}$ and $N_{\mathcal{H}}(v_j)\subseteq \mathcal{F}\setminus \{e_t\}$.

For the second statement, we let $v\in (e_i\cap e_j)\setminus U$. 
We consider the following Berge paths.
\begin{equation*}
\begin{aligned}
P_C=&e_{i-1},v_{i-2},e_{i-2},v_{i-3},\ldots,v_j,e_j,v,e_i,v_i,e_{i+1},\ldots,\\
&e_{t-1},v_{t-1},e_t,v_{j-1},e_{j-1},v_{j-2},\ldots,v_1,e_1. 
\end{aligned}
\end{equation*}
\begin{equation*}
\begin{aligned}
P_D=&e_{j+1},v_{j+1},e_{j+2},v_{j+2},\ldots,e_{i-1},v_{i-1},e_t,v_{t-1},e_{t-1},\ldots,\\
&e_{i+1},v_i,e_i,v,e_j,v_{j-1},e_{j-1},\ldots,v_1,e_1. 
\end{aligned}
\end{equation*}
Applying Lemma~\ref{lemma21} for $P_C$ and $P_D$, we have that $N_{\mathcal{H}}(v_{i-1})\subseteq \mathcal{F}\setminus \{e_1\}$ and $N_{\mathcal{H}}(v_j)\subseteq \mathcal{F}\setminus \{e_1\}$.
\end{proof}

Let $d_1\ge 1$ and $d_2\ge 1$ be two integers with $d_1+d_2\le m+1$ such that $V(e_1)=V(e_2)=\cdots=V(e_{d_1})\neq V(e_{d_1+1})$ and $V(e_t)=V(e_{t-1})=\cdots=V(e_{t-d_2+1})\neq V(e_{t-d_2})$.
\begin{claim}
If $e_1\cap U=\{v_1,v_2,\ldots,v_{d_1}\}$ and $e_t\cap U=\{v_{t-1},v_{t-2},\ldots,v_{t-d_2}\}$, then either $(e_1\cup e_t)\setminus \{v_{d_1},v_{t-d_2}\}$ is incident with $m+1$ hyperedges or there exists a set $S$ of size at least $2r-1$ such that $N_{\mathcal{H}}(S)\subseteq \mathcal{F}$.
\end{claim}
\begin{proof}
First of all, note that $(e_1\cap e_t)\setminus U=\emptyset$ since $\mathcal{H}$ is $\mathcal{B}C_{t}$-free. Note that $v_1,v_2,\ldots,v_{d_1-1}$ can be exchanged with the vertices in $e_1\setminus U$ and $v_{t-1},v_{t-2},\ldots,v_{t-d_2+1}$ can be exchanged with the vertices in $e_t\setminus U$.
Since $\mathcal{H}$ is $\mathcal{B}C_t$-free, we have $\{v_1,v_2,\ldots,v_{d_1-1}\}\cap \{v_{t-1},v_{t-2},\ldots,v_{t-d_2+1}\}=\emptyset$.

By Lemma~\ref{lemma21}, we have $N_{\mathcal{H}}(e_1\setminus \{v_{d_1}\})\subseteq \mathcal{F}\setminus \{e_t\}$ and $N_{\mathcal{H}}(e_t\setminus \{v_{t-d_2}\})\subseteq \mathcal{F}\setminus \{e_1\}$. 
Suppose $w_1\in e_1\setminus \{v_{d_1}\}$ is incident with a hyperedge $e_j$ ($d_1<j\le t-d_2$). Then the Berge path 
\[P_A=e_{j-1},v_{j-2},e_{j-2},v_{j-3},\ldots,e_2,v_1,e_1,w_1,e_j,v_j,\ldots,e_{t-1},v_{t-1},e_t
\]
has maximal length. 
Since $v_{j-1}$ is a non-defining vertex in the first hyperedge of $P_A$, applying Lemma~\ref{lemma21} to $P_A$, we have that $N_{\mathcal{H}}(v_{j-1})\subseteq \mathcal{F}\setminus \{e_{t-d_2+1},e_{t-d_2+2},\ldots,e_{t}\}$.
Similarly, suppose $w_2\in e_t\setminus \{v_{d_{t-d_2}}\}$ is incident with a hyperedge $e_k$ ($d_1+1\le k\le t-d_2$), we may assume $w_2\in e_t\setminus U$. Then the Berge path 
\[
P_B=e_{k+1},v_{k+1},e_{k+2},v_{k+2},\ldots,e_{t-1},v_{t-1},e_t,w_2,e_k,v_{k-1},\ldots,v_1,e_1
\]
has maximal length. 
Since $v_{k}$ is a non-defining vertex in the first hyperedge of $P_B$, applying Lemma~\ref{lemma21} to $P_B$, we have that $N_{\mathcal{H}}(v_{k})\subseteq \mathcal{F}\setminus \{e_1,e_2,\ldots,e_{d_1}\}$. 
Thus, we have that $S=(e_1\setminus \{v_{d_1}\}\cup \{v_{j-1}\})\cup (e_t\setminus \{v_{t-d_2}\}\cup \{v_{k}\})$ of size at least $2r-1$ such that $N_{\mathcal{H}}(S)\subseteq \mathcal{F}$. 
Otherwise,
if there are no such $w_1$ and $w_2$, then we have a set $(e_1\cup e_t)\setminus \{v_{d_1},v_{t-d_2}\}$ of size at least $2r-2$ incident with at most $d_1+d_2\le m+1$ hyperedges. 

If there exists $w_1\in e_1\setminus {v_{d_1}}$ such that $w_1$ is incident with one hyperedge $e_{j_1}$ satisfying ($d_1<j_1\le t-d_2$) but for each vertex $w_2\in e_t\setminus \{v_{t-d_2}\}$ we have $N_{\mathcal{H}}(w_2)\subseteq \{e_{t-d_2+1},e_{t-d_2+2},\ldots,e_{t}\}$,
then by the above arguments we have a set $S=(e_1\setminus \{v_{d_1}\}\cup \{v_{j_1-1}\})\cup (e_t\setminus \{v_{t-d_2}\})$ of size at least $2r-1$ such that $N_{\mathcal{H}}(S)\subseteq \mathcal{F}$.

If there exists $w_2\in e_t\setminus {v_{t-d_2}}$ such that $w_2$ is incident with one hyperedge $e_{j_2}$ satisfying ($d_1<j_2\le t-d_2$) but for each vertex $w_1\in e_1\setminus \{v_{d_1}\}$ we have $N_{\mathcal{H}}(w_1)\subseteq \{e_1,e_2,\ldots,e_{d_1}\}$, then we have a set $S=(e_1\setminus \{v_{d_1}\})\cup (e_t\setminus \{v_{t-d_2}\}\cup \{v_{j_2}\})$ of size at least $2r-1$ such that $N_{\mathcal{H}}(S)\subseteq \mathcal{F}$.

This completes this claim.
\end{proof}

From here we may assume that $|(e_1\cup e_t)\cap U|>d_1+d_2$. 
Let $e_1\cap U=\{v_{i_0},v_{i_1},\ldots,v_{i_s}\}$ and $e_t\cap U=\{v_{j_{s+1}-1},v_{j_{s+2}-1},\ldots,v_{j_{\ell}-1}\}$, where $1=i_0<i_1<i_2<\cdots<i_s$ and $j_{s+1}-1<j_{s+2}-1<\cdots<j_{\ell}-1=t-1$.

Recursively define the sets $A_1:=(e_1\cup e_t)\setminus U$ and for $p=1,2,\ldots,s$,
\begin{equation*}
A_{p+1}=
\begin{cases}
A_p\cup (e_{i_p}\setminus U), &\mbox{if }(e_{i_p}\setminus U)\cap A_p=\emptyset;\\
A_p\cup (e_{i_p}\setminus U)\cup \{v_{i_p-1}\},&\mbox{otherwise}.
\end{cases}
\end{equation*}
During this process, we always have $(e_{i_p}\setminus U)\cap (e_t\setminus U)=\emptyset$ for any $p=1,2,\ldots,s$. If not, we assume $w\in (e_{i_p}\setminus U)\cap (e_t\setminus U)$. 
Then,
$$
v_{t-1},e_t,w,e_{i_p},v_{i_p-1},e_{i_p-1},\ldots,v_1,e_1,v_{i_p},e_{i_p+1},v_{i_p+1},\ldots,e_{t-1},v_{t-1}
$$
is a Berge cycle of length $t$, a contradiction.

For $p=s+1,s+2,\ldots,\ell$,
\begin{equation*}
A_{p+1}=
\begin{cases}
A_p\cup (e_{j_p}\setminus U), &\mbox{if }(e_{j_p}\setminus U)\cap A_p=\emptyset;\\
A_p\cup (e_{j_p}\setminus U)\cup \{v_{j_p-1}\},&\mbox{otherwise}.
\end{cases}
\end{equation*}
During this process, we always have $(e_{j_p}\setminus U)\cap (e_1\setminus U)=\emptyset$ for any $p=s+1,s+2,\ldots,\ell$. If not, we assume $w\in (e_{j_p}\setminus U)\cap (e_1\setminus U)$. Then,
$$
v_1,e_1,w,e_{j_p},v_{j_p},e_{j_p+1},v_{j_p+1},\ldots,e_{t-1},v_{t-1},e_t,v_{j_p-1},e_{j_p-1},v_{j_p-2},\ldots,e_2,v_1
$$
is a Berge cycle of length $t$, a contradiction.

Note that for $1\le p\le s$, the only possible defining vertices in $A_p$ are $v_{i_1-1},v_{i_2-1},\ldots,v_{i_{p-1}-1}$. 
Therefore $v_{i_p-1}$ is not contained in $A_p$.
For $s+1\le p\le \ell$, the only possible defining vertices in $A_p$ are $v_{i_1-1},v_{i_2-1},\ldots,v_{i_{s-1}-1},\ldots,v_{j_{s+1}-1},v_{j_{s+2}-1},\ldots,v_{j_{p-1}-1}$, therefore $v_{j_p-1}$ is not contained in $A_p$. Next we will show that the defining vertices in $A_p$ are distinct for $s+1\le p\le \ell$.
\begin{claim}
There are no indices $1\le p\le s$ and $s+1\le q\le \ell$ such that $v_{i_p-1}=v_{j_q-1}$, where $v_{i_p-1}\in A_{p+1}$ and $v_{j_q-1}\in A_{q+1}$.
\end{claim}
\begin{proof}
Assume that there exist $1\le p\le s$ and $s+1\le q\le \ell$ such that $v_{i_p-1}=v_{j_q-1}$, where $v_{i_p-1}\in A_{p+1}$ and $v_{j_q-1}\in A_{q+1}$. 

Since $v_{i_p-1}\in A_{p+1}$, there must exist an index $1\le p'\le p-1$ such that $(e_{i_{p'}}\setminus U)\cap (e_{i_p}\setminus U)\neq \emptyset$. Note that $v_{i_{p'}},v_{i_p}\in e_1\cap U$ and $v_{i_p-1}\in e_t\cap U$. 
If we suppose that $w\in (e_{i_{p'}}\setminus U)\cap (e_{i_{p}}\setminus U)$, then
\begin{equation*}
\begin{aligned}
P'=&v_{i_p},e_{i_p+1},v_{i_p+1},\ldots,v_{t-1},e_t,v_{i_p-1},e_{i_p-1},v_{i_p-2},\ldots,\\
&e_{i_{p'}+1},v_{i_{p'}},e_1,v_1,e_2,v_2,\ldots,v_{i_{p'}-1},e_{i_{p'}},w,e_{i_p},v_{i_p}
\end{aligned}
\end{equation*}
is a Berge cycle of length $t$, a contradiction.
\end{proof}

Let us denote $A=A_{\ell+1}$. Then, we have 
$|A_p|<|A_{p+1}|$ for all $1\le p\le \ell$. Hence, $|A|\ge |A_1|+\ell\ge 2r-1$, by Lemmas \ref{lemma21}, \ref{lemma22} and \ref{lemma23} we have that $|N_{\mathcal{H}}(A)|\le k-1$. 

This completes the proof.

\section{Concluding Remarks}\label{seccr}
In this paper, we determined $\ex^{\mathrm{conn}}_r(n,\mathcal{B}P_k)$ when $n$ is sufficiently large and $n$ is not a multiple of $r$. For the case $k=r+1$, we determined $\ex^{\mathrm{conn}}_r(n,\mathcal{B}P_k)$ asymptotically. We conjecture that the behavior of the function $\ex^{\mathrm{conn}}_r(n,\mathcal{B}P_k)$ for the case $r+2\le k\le 2r-1$ will be very similar to that for $k=r+1$. 
\begin{conjecture}
Fix integers $k$ and $r$ such that $r+1\le k\le 2r-1$. Then for sufficiently large $n$,
$$
\ex^{\mathrm{conn}}_r(n,\mathcal{B}P_k)=n-(k-2)+\binom{k-2}{r}.
$$
\end{conjecture}
Indeed, we can construct an extremal $r$-uniform hypergraph $\mathcal{H}$ on $n$ vertices as follows. First, let $S\subseteq V(\mathcal{H})$ be a vertex subset
of $\mathcal{H}$ containing $k-2$ vertices such that $S$ forms a complete $r$-uniform subhypergraph in $\mathcal{H}$. For each vertex $v\in V(\mathcal{H})\setminus S$, choose an $(r-1)$-subset in $S$ such that $S\cup \{v\}$ forms a hyperedge in $\mathcal{H}$. It can be easily checked that $\mathcal{H}$ is $\mathcal{B}P_k$-free and connected.

\section*{Ackowledgements}
The research of Wang and Zhang was supported by the National Natural Science Foundation of
China (No. 12271439) and China Scholarship Council (No. 202206290003).
The research of Gy\H{o}ri was supported by NKFIH grant K132696.
The research of Tompkins was supported by NKFIH grant K135800.


\begin{thebibliography}{99}
\bibitem{BGLS} P.N. Balister, E. Gy\H{o}ri, J. Lehel, R.H. Schelp, Connected graphs without long paths, Discrete Math. 308 (19) (2008)
4487--4494.

\bibitem{DGMT} A. Davoodi, E. Győri, A. Methuku, C. Tompkins, An Erd\H{o}s-Gallai type theorem for uniform hypergraphs,
European J. Combin. 69 (2018) 159--162.

\bibitem{EG} P. Erd\H{o}s, T. Gallai, On the maximal paths and cricuits of graphs, Acta Math. Acad. Sci. Hung. 10 (1959) 337--357.

\bibitem{ES-1} P. Erd\H{o}s, M. Simonovits, A limit theorem in graph theory
Studia Sci. Math. Hungar. 1 (1966) 51--57.

\bibitem{ES-2} P. Erd\H{o}s, A.H. Stone, On the structure of linear graphs,
Bull. Amer. Math. Soc. 52 (1946) 1089--1091

\bibitem{EGMSTZ} B. Ergemlidze, E. Gy\H{o}ri, A. Methuku, N. Salia, C. Tompkins, O. Zamora
Avoiding long Berge cycles, the missing cases $k=r+1$ and $k=r+2$,
Comb. Probab. Comput. (2020) 1--13. 

\bibitem{FS} R.J. Faudree, R.H. Schelp, Path Ramsey numbers in multicolorings, J. Combin. Theory Ser. B 19 (1975) 150--160.

\bibitem{FKL19} Z. F\"{u}redi, A. Kostochka, R. Luo, Avoiding long Berge-cycles, J. Combin. Theory Ser. B 137 (2019) 55--64.

\bibitem{FKL} Z. F\"{u}redi, A. Kostochka, R. Luo, On 2-connected hypergraphs with no long cycles, Electron. J. Combin. 26 (4) (2019) P4.31.

\bibitem{FKL21} Z. F\"{u}redi, A. Kostochka, R. Luo, Avoiding long Berge cycles II, exact bounds for all $n$, Journal of Combinatorics 12 (2021) 247--268.

\bibitem{FKV} Z. F\"{u}redi, A. Kostochka, J. Verstra\"{e}te, Stability Erd\H{o}s–Gallai theorems on cycles and paths, J. Combin. Theory Ser. B 121 (2016) 197--228.

\bibitem{FS1} Z. F\"{u}redi, M. Simonovits, The history of degenerate (bipartite) extremal graph problems, in: Erd\H{o}s Centennial,
Springer, Berlin, Heidelberg, 2013, pp. 169--264.


\bibitem{GNPSV} D. Gerbner, D.T. Nagy, B. Patk\'{o}s, N. Salia, M. Vizer,
Stability of extremal connected hypergraphs avoiding Berge-paths,
European J. Combin. 118 (2024) 103930.

\bibitem{GKL} E. Gy\H{o}ri, G. Katona, N. Lemons, Hypergraph extensions of the Erd\H{o}s-Gallai theorem, European J. Combin. 58 (2016) 238--246.

\bibitem{GLSZ} E. Gy\H{o}ri, N. Lemons, N. Salia, O. Zamora, The structure of hypergraphs without long Berge cycles, J. Combin. Theory Ser. B 148 (2021) 239--250.

\bibitem{GMSTV} E. Gy\H{o}ri, A. Methuku, N. Salia, C. Tompkins, M. Vizer, On the maximum size of connected hypergraphs without a
path of given length, Discrete Math. 341 (9) (2018) 2602--2605.

\bibitem{GSZ21} E. Gy\H{o}ri, N. Salia, O. Zamora, Connected hypergraphs without long Berge-paths, European J. Combin. 96 (2021) 103353.

\bibitem{Kopylov} G.N. Kopylov, On maximal paths and cycles in a graph, Sov. Math. (1977) 593--596.

\bibitem{KL} A. Kostochka, R. Luo, On $r$-uniform hypergraphs with circumference less than $r$, Discrete Appl. Math. 276 (2020) 69--91.
\end{thebibliography}
\end{document}